\documentclass[review]{elsarticle}
\usepackage{graphicx, epstopdf}
\usepackage{amssymb,amsmath}
\usepackage[toc,page,title,titletoc,header]{appendix}
\usepackage{amsbsy}
\usepackage{amsthm}
\usepackage{multirow}
\usepackage{indentfirst}
\usepackage{natbib}
\usepackage{amscd}
\numberwithin{equation}{section}

\newtheorem{Theorem}{Theorem}[section]

\newtheorem{Lemma}{Lemma}[section]

\newtheorem{Assumption-Notation}[Theorem]{Assumption-Notation}

\newtheorem{Proposition}{Proposition}[section]
\newtheorem{Remark}{Remark}[section]

\long\def\symbolfootnote[#1]#2{\begingroup
\def\thefootnote{\fnsymbol{footnote}}\footnote[#1]{#2}\endgroup}

\begin{document}
\begin{frontmatter}
\title{Alternative approach to derive q-potential measures of refracted spectrally L\'evy processes}
 \author{Jiang Zhou\corref{cor1}}
 \ead{1101110056@pku.edu.cn}
 \author{Lan Wu\corref{}}
 \ead{lwu@pku.edu.cn}
\cortext[cor1]{Corresponding author.}

\address{School of Mathematical Sciences, Peking University, Beijing 100871, P.R.China}

\begin{abstract}{
For a refracted L\'evy process driven by a spectrally negative L\'evy process, we use a different approach to derive expressions for its q-potential measures without killing. Unlike previous methods whose derivations depend on scale functions which are defined only for spectrally negative L\'evy processes, our approach is free of scale functions. This makes it possible to extend the result here to a quite general refracted L\'evy process by applying the approach presented below.
}
\end{abstract}

\end{frontmatter}

\section{Introduction}
A refracted L\'evy process $U=(U_t)_{t\geq 0}$, proposed in~\cite{KL}, is a stochastic process satisfying
\begin{equation}
U_t=X_t-\delta \int_{0}^{t}\textbf{1}_{\{U_s > b\}}ds,
\end{equation}
where $\delta, b \in \mathbb R$, $X=(X_t)_{t\geq 0}$ is a L\'evy process and for a given set $A$,   $\textbf{1}_{A}$ is the indicator function. Several papers have investigated refracted L\'evy processes and many results have been derived. For example, the expression of q-potential measures of $U$ without killing (i.e., $\mathbb E_x\left[\int_{0}^{\infty}e^{-q t}\rm{\bf{1}}_{\{U_t \in dy\}}dt\right]$) is deduced, see Theorem 6 (iv) in~\cite{KL} for the case of $X$ being a
spectrally negative L\'evy process and Corollary 4.1 in~\cite{Zhouthe} for the case that $X$ is a jump diffusion process with jumps having rational Laplace transforms. Formulas for the joint Laplace transform of $(\nu_a^-, \int_{0}^{\nu_a^-}\textbf{1}_{\{U_s< b\}}ds)$ (here $\nu_a^-$ is the first passage time of $U$) are also obtained, see~\cite{R} for the case that $X$ is a L\'evy process without positive jumps  and~\cite{ZhouSPL} for the case of $X$ being a double exponential jump diffusion process. For other results, see~\cite{KPP,LCR}. A very similar process to $U$ has been studied in~\cite{ZhouIME} under a hyper-exponential jump diffusion model, see also~\cite{Wu} for a more general model.

To investigate refracted L\'evy processes, the following question needs to be considered first: Is there a
strong solution $U$ to (1.1) for a given L\'evy process $X$? If $X$ has non-zero Gaussian component, some results for the existence of a solution to (1.1) have been derived, the reader is referred to, e.g.,~\cite{St}. If $X$ is a spectrally negative L\'evy process with some limitations satisfied, the authors in~\cite{KL} have showed that there is a unique strong solution $U$ to (1.1) by applying the following method: they first investigate  the case of $X$ with bounded variation, and then consider the case that $X$ has paths of unbounded variation.

In this paper, another approach is established to show that there is a strong solution to (1.1) if $X$ is a spectrally negative L\'evy process and satisfies some additional conditions. Different from the discussion presented in~\cite{KL}, we start with a situation under which $X$ is of unbounded variation and then study the general case via an approximation argument. It is commented that some derivations in~\cite{KL} depend heavily on the scale functions which are defined only for spectrally negative L\'evy processes, thus they cannot be extended to the case that $X$ has both positive and negative jumps. However, our approach is free of scale functions and still works for a general L\'evy process.
It is hoped that our approach and idea will play a role in future research on refracted L\'evy processes.

For a refracted L\'evy process $U$ driven by a spectrally negative L\'evy process $X$, in~\cite{KL}, the authors have derived the expression of $\mathbb E_x\left[\int_{0}^{\infty}e^{-q t}\rm{\bf{1}}_{\{U_t \in dy\}}dt\right]$. In this paper,  we obtain another formula for $\mathbb E_x\left[\int_{0}^{\infty}e^{-q t}\rm{\bf{1}}_{\{U_t \in dy\}}dt\right]$. Unlike the results in~\cite{KL}, which are written in terms of scale functions, our results are written in terms of Wiener-Hopf factors (see (2.15) below). In~\cite{Zhouthe}, we have showed that (2.15) holds for a refracted L\'evy process driven by a jump diffusion process with jumps of rational Laplace transforms and proposed a conjecture: Formula (2.15) holds for a more general refracted L\'evy process. Theorem 2.2 below confirms that the above conjecture is correct for a refracted spectrally negative L\'evy process. Based on the ideas in this paper, we want to prove the above mentioned conjecture in our future research.

The remainder of this paper is organized as follows. In Section 2, we give some notations and the main results. Then in Section 3, we give some preliminary results and prove the main results. Finally, some further discussions and remarks are presented in Section 4.

\section{Notations and main results}
\subsection{Notations}
In this section and the next one, the process $X=(X_t)_{t\geq 0}$ in (1.1) is a general spectrally negative L\'evy process. As in~\cite{KL}, for given $\delta \in \mathbb R$,
introduce the following process $Y=\{Y_t=X_t-\delta t;t \geq 0\}$. Define $\underline{X}_{T}:=\inf_{0\leq t \leq T}X_t$ and $\overline{X}_{T}:=\sup_{0\leq t \leq T}X_t$ for any given $T>0$. For the process $Y$, we have similar definitions.

The law of $X$ ($Y$) starting from $x$ ($y$) is denoted by $\mathbb P_x$ ($\hat{\mathbb P}_y$), and $\mathbb E_x$ ($\hat{\mathbb E}_y$) denotes the corresponding expectation, and when $x = 0$ ($y=0$), we write $\mathbb P$ ($\hat{\mathbb P}$) and $\mathbb E$ ($\hat{\mathbb E}$) to simplify the notation. For given $q > 0$, the variable $e(q)$, independent of all other stochastic processes,
is an exponential random variable, whose expectation equals to $\frac{1}{q}$.

For a spectrally negative L\'evy process $X$, it is known that its Laplace exponent $\psi(z)$ is given by
\begin{equation}
\begin{split}
\psi(z)=\ln (\mathbb E\left[e^{z X_1}\right])
&=\frac{1}{2}\sigma^2 z^2+ \gamma z -\int_{(1,\infty)}(1-e^{- zx})\Pi(dx)\\
&-\int_{(0,1)}(1-e^{-z x}-z x)\Pi(dx), \ \ z \geq 0,
\end{split}
\end{equation}
where $\gamma \in \mathbb R$, $\sigma \geq 0$ and the L\'evy measure $\Pi$ satisfies $\int_{0}^{\infty}(x^2\wedge 1)\Pi(dx)< \infty$. If $\sigma=0$ and $\int_{(0,1)}x\Pi(dx)<\infty$, then
$X$ has paths of bounded variation and the Laplace exponent $\psi(\theta)$ can be written as
\begin{equation}
\psi(\theta)= d \theta -\int_{0}^{\infty}(1-e^{-\theta x})\Pi(dx),
\end{equation}
where $d=\gamma + \int_{(0,1)} x\Pi(dx)>0$. As in~\cite{KL}, we require that $d>\delta >0$ if $X$ is of bounded variation. For the sake of convenience, denote by $\mathbb S\mathbb N\mathbb L\mathbb P$ the space of all spectrally negative L\'evy processes $X$ with unbounded variation or bounded variation but satisfying $d>\delta >0$.

Similar to~\cite{KL}, for any given $q> 0$, we define
\begin{equation}
\Phi(q)=\sup\{\theta \geq 0: \psi(\theta)=q\} \ \ and \ \ \varphi(q)=\sup\{\theta \geq 0: \psi(\theta)-\delta \theta =q\},
\end{equation}
where $\psi(\theta)$ is given by (2.1).
Note that $\psi(\theta)-\delta \theta$ is the Laplace exponent of $Y$, i.e., $\psi(\theta)-\delta \theta=\ln (\hat{\mathbb E}\left[e^{\theta Y_1}\right])$ for $\theta \geq 0$.

For $q>0$ and $s \geq 0$, we know (see, e.g., (8.2) in~\cite{K})
\begin{equation}
\mathbb E\left[e^{-s\overline{X}_{e(q)}}\right]=\frac{\Phi(q)}{\Phi(q)+s}, \ \ \mathbb E\left[e^{s\underline{X}_{e(q)}}\right]=\frac{q}{\Phi(q)}\frac{\Phi(q)-s}{q-\psi(s)},
\end{equation}
and
\begin{equation}
\hat{\mathbb E}\left[e^{-s\overline{Y}_{e(q)}}\right]=\frac{\varphi(q)}{\varphi(q)+s}, \ \
\hat{\mathbb E}\left[e^{s\underline{Y}_{e(q)}}\right]=\frac{q}{\varphi(q)}\frac{\varphi(q)-s}{q-(\psi(s)-\delta s)},
\end{equation}
where $\Phi(q)$ and $\varphi(q)$ are given by (2.3).

\begin{Remark}
For a given L\'evy process $X$ and $q>0$, the quantity $\mathbb E\left[e^{-s\overline{X}_{e(q)}}\right]$
$\left(\mathbb E\left[e^{s\underline{X}_{e(q)}}\right]\right)$ is called the positive (negative) Wiener-Hopf factor of $X$.
\end{Remark}

Introduce first the following functions. $K_q(x)$ is the convolution of $\underline{Y}_{e(q)}$ under $\hat{\mathbb P}$ and $\overline{X}_{e(q)}$ under $\mathbb P$, i.e.,
\begin{equation}
K_q(x)=\int_{-\infty}^{\min\{0,x\}}\mathbb P\left(\overline{X}_{e(q)} \leq x-z\right)\hat{\mathbb P}\left(\underline{Y}_{e(q)} \in dz \right), \ \ x \ \in \mathbb R,
\end{equation}
$F_1(x)$ is a  function on $(0, \infty)$ with the Laplace transform:
\begin{equation}
\int_{0}^{\infty} e^{-s x}F_1(x)dx
=\frac{1}{s}\left(\frac{\hat{\mathbb E}\left[e^{-s \overline{Y}_{e(q)}}\right]}
{\mathbb E\left[e^{-s \overline{X}_{e(q)}}\right]}-1\right), \ \ s > 0,
\end{equation}
and $F_2(x)$ is a  function on $(-\infty, 0)$ with the Laplace transform:
\begin{equation}
\begin{split}
&\int_{-\infty}^{0} e^{sz}F_2(z)dz
=\frac{1}{s}\left(\frac{\mathbb E\left[e^{s \underline{X}_{e(q)}}\right]}
{\hat{\mathbb E}\left[e^{s \underline{Y}_{e(q)}}\right]}-1\right), \ \ s > 0,
\end{split}
\end{equation}
Define
\[
F_1(0):=\lim_{x \downarrow 0}F_1(x), \ \ F_2(0):=\lim_{x \uparrow 0}F_2(x),
\]
and
\begin{equation}
F_1(\infty):=\lim_{x \uparrow \infty}F_1(x)=\lim_{s \downarrow 0}\int_{0}^{\infty}se^{-s x}F_1(x)dx=0= F_2(-\infty):=\lim_{x \downarrow -\infty}F_2(x).
\end{equation}

Formulas (2.4) and (2.5) give
\begin{equation}
\lim_{s \uparrow \infty}\frac{\hat{\mathbb E}\left[e^{-s \overline{Y}_{e(q)}}\right]}
{\mathbb E\left[e^{-s \overline{X}_{e(q)}}\right]}=\lim_{s \uparrow \infty}\frac{\varphi(q)}{\varphi(q)+s}\frac{\Phi(q)+s}{\Phi(q)}=
\frac{\varphi(q)}{\Phi(q)}.
\end{equation}
It is obvious that $\lim_{\theta \uparrow \infty}\frac{\psi(\theta)}{\theta}=\infty$ if $\sigma > 0$. If $\int_{(0,1)}x\Pi(dx)=\infty$, we note first that $\theta x +e^{-\theta x}-1 \geq 0$ for $0<x <1$ and $\theta \geq 0$, and obtain (by applying the Fatou lemma)
\[
\lim_{\theta \uparrow \infty}\frac{-\int_{(0,1)}(1-e^{-\theta x}-\theta x)\Pi(dx)}{\theta} \geq
\int_{(0,1)} x\Pi(dx)= \infty.
\]
These results and formulas (2.4) and (2.5) will produce
\begin{equation}
\begin{split}
\lim_{s \uparrow \infty} \frac{ \mathbb E\left[e^{s\underline{X}_{e(q)}}\right]}{ \hat{\mathbb E}\left[e^{s\underline{Y}_{e(q)}}\right]}
&=\lim_{s \uparrow \infty} \frac{\varphi(q)}{q}\frac{q-(\psi(s)-\delta s)}{\varphi(q)-s}\frac{q}{\Phi(q)}\frac{\Phi(q)-s}{q-\psi(s)}\\
&=\frac{\varphi(q)}{\Phi(q)}\lim_{s \uparrow \infty} \frac{q-(\psi(s)-\delta s)}{q-\psi(s)}\\
&=\left\{\begin{array}{cc}
\frac{\varphi(q)}{\Phi(q)}, &\ if \ \sigma>0 \ or\ \int_{(0,1)}x\Pi(dx)=\infty,  \\
\frac{\varphi(q)}{\Phi(q)}\frac{d-\delta}{d}, &\ \ otherwise,
\end{array}\right.
\end{split}
\end{equation}
where $d$ is given by (2.2).

Noting that $F_1(0)=\lim_{s \uparrow \infty} \int_{0}^{\infty}
se^{-sx}F_1(x)dx$, we can obtain from (2.7) and (2.10) that
\begin{equation}
F_1(0)=\frac{\varphi(q)}{\Phi(q)}-1.
\end{equation}
Similarly, we can derive that
\begin{equation}
F_2(0)=\left\{\begin{array}{cc}
\frac{\varphi(q)}{\Phi(q)}-1, &\ if \ \sigma>0 \ or\ \int_{(0,1)}x\Pi(dx)=\infty,  \\
\frac{\varphi(q)}{\Phi(q)}\frac{d-\delta}{d}-1, &\ \ otherwise.
\end{array}\right.
\end{equation}

\begin{Remark}
It follows from (2.12) and (2.13) that
\[
\begin{split}
\left\{\begin{array}{cc}
F_1(0)=F_2(0), &\ if \ \sigma>0 \ or\ \int_{(0,1)}x\Pi(dx)=\infty,  \\
F_1(0)\neq F_2(0)\ \ if \ \ \delta \neq 0, &\ \ otherwise.
\end{array}\right.
\end{split}
\]
\end{Remark}

\subsection{Main results}
The main purpose of this paper is to show the following Theorems 2.1 and 2.2, and the details on the derivation are left to Section 3.

\begin{Theorem}
For any $X \in \mathbb S\mathbb N\mathbb L\mathbb P$, equation (1.1) exists a strong solution $U=(U_t)_{t\geq 0}$. Moreover,
\begin{equation}
\mathbb P_x\left(U_{e(q)}=y\right)=0, \ \ for \ \ all \ \  q>0 \  and \ \ y \in  \mathbb R.
\end{equation}
\end{Theorem}

\begin{Remark}
The above strong solution $U$ is unique if $\delta >0$ in (1.1), and the reader is referred to Proposition 15 in~\cite{KL} for the proof. But, if $\delta <0$, we cannot find an easy approach to show that $U$ is the unique solution to (1.1).
\end{Remark}

\begin{Theorem}
For any given $X \in \mathbb S\mathbb N\mathbb L\mathbb P$ and the associated strong solution $U$ given by Theorem 2.1, we have
\begin{equation}
\begin{split}
&\mathbb P_x\left(U_{e(q)} \in  dy\right)=q\int_{0}^{\infty}e^{-q t}\mathbb P_x\left(U_{t} \in  dy\right)dt=\\
&\left\{\begin{array}{cc}
\big(F_1(0)+1\big)K_q(dy-x)+\int_{b-x}^{y-x}F_1(dy-x-z)K_q(dz),&y \geq b,\\
\big(F_2(0)+1\big)K_q(dy-x)- \int_{y-x}^{b-x}F_2(dy-x-z)K_q(dz), & y < b,
\end{array}\right.
\end{split}
\end{equation}
where $K_q(x)$, $F_1(x)$ and $F_2(x)$ are given by (2.6), (2.7) and (2.8), respectively.
\end{Theorem}

\begin{Remark}
We have to mention that formulas for $\mathbb P_x\left(U_{e(q)} \in  dy\right)$ and Theorem 2.1 have
been obtained in~\cite{KL} $($see Theorem 1 and Theorem 6 (iv) in that paper$)$. Here we use a different approach to establish these results. Our approach is useful and may be used to establish the following result: Theorems 2.1 and 2.2 hold for a general L\'evy process $X$ $($see Section 4 for more discussions$)$.
\end{Remark}

\begin{Remark}
Expressions of $\mathbb P_x\left(U_{e(q)} \in  dy\right)$ given by Theorem 6 (iv) in~\cite{KL} are written in terms of scale functions and are completely different from (2.15). For the case of $\delta >0$, we have confirmed in~\cite{Zhouanote} that formula (2.15) and Theorem 6 (iv) in~\cite{KL} are the same. If $\delta <0$, then $\varphi(q)<\Phi(q)$ (see (2.3)). This means that Theorem 6 (iv) in~\cite{KL} does not hold for $\delta <0$ since it contains the following quantity: $\int_{0}^{\infty}e^{-\varphi(q)z}
W^{(q)\prime}(z-y)dz$, which equals to $\infty$ as $\lim_{x \uparrow \infty} \frac{W^{(q)\prime }(x)}{e^{\Phi(q)x}}=\Phi(q)\Phi^{\prime}(q)>0$. Here, $W^{(q)}(x)$ is the q-scale function of $X$ and $y \in\mathbb R$ is given. The interested reader is referred to~\cite{KL} and chapter $8$ in~\cite{K} for more details.
\end{Remark}

At the end  of this subsection, some properties of $K_q(x)$, $F_1(x)$ and $F_2(x)$ are given in the following theorem, which will play an important role in the proof of Theorems 2.1 and 2.2.

\begin{Theorem}
For $X$ in $\mathbb S\mathbb N\mathbb L\mathbb P$, the following results hold.

(i) If $\delta >0$ $(\delta <0)$, then $F_1(x)$ is decreasing (increasing) on $[0, \infty)$ and satisfies
\begin{equation}
\int_{0}^{\infty} e^{-s x}d F_1(x)=\frac{\hat{\mathbb E}\left[e^{-s \overline{Y}_{e(q)}}\right]}
{\mathbb E\left[e^{-s \overline{X}_{e(q)}}\right]}-\frac{\varphi(q)}{\Phi(q)}, \ \ s\geq 0.
\end{equation}
Besides, for $x \in \mathbb R$ and $\delta \in \mathbb R$, it holds that
\begin{equation}
-1 \leq F_1(x) \leq 0\vee \frac{\varphi(q)-\Phi(q)}{\Phi(q)}.
\end{equation}

(ii) If $\delta >0$ $(\delta <0)$, then $e^{\varphi(q)x}F_2(x)$ is increasing (decreasing) on $(-\infty,0]$ and satisfies
\begin{equation}
\int_{-\infty}^{0} e^{sx} d (e^{\varphi(q)x}F_2(x))=
\frac{s}{s+\varphi(q)}\left(1-\frac{\mathbb E\left[e^{(s+\varphi(q)) \underline{X}_{e(q)}}\right]}
{\hat{\mathbb E}\left[e^{(s+\varphi(q)) \underline{Y}_{e(q)}}\right]}\right)+F_2(0), \ \ s\geq 0.
\end{equation}
In addition, we have
\begin{equation}
-1 \leq e^{\varphi(q)x}F_2(x) \leq 0 \vee\frac{\varphi(q)-\Phi(q)}{\Phi(q)}.
\end{equation}

(iii) For $\delta \in \mathbb R$, both $F_1(x)$ and $e^{-\varphi(q)x}F_2(-x)$ are continuous on $[0,\infty)$.

(iv) $K_q(x)$ is continuous on $(-\infty,\infty)$.
\end{Theorem}

\begin{proof}

(i) From (2.4), (2.5) and (2.7), some straightforward calculations give
\begin{equation}
F_1(x)=\frac{\varphi(q)-\Phi(q)}{\Phi(q)}e^{-\varphi(q)x}, \ \ x>0.
\end{equation}
Noting that $\varphi(q)<\Phi(q)$ if $\delta <0$ and $\varphi(q)>\Phi(q)$ if $\delta >0$, we obtain that
$F_1(x)$ is decreasing (increasing) on $[0, \infty)$ if $\delta >0$ $(\delta <0)$. From (2.7) and (2.12), we can derive (2.16) by applying integration by parts. Formula (2.17) follows directly from (2.20).

(ii) Formulas (2.4), (2.5) and (2.8) yield
\[
\begin{split}
&\int_{-\infty}^{0}e^{sx}F_2(x)dx=\frac{1}{s}\left(\frac{\Phi(q)-s}{\Phi(q)(q-\psi(s))}
\frac{\varphi(q)(q-\varphi(s)+\delta s)}{\varphi(q)-s}-1\right)\\
&=\frac{\delta \varphi(q)}{\Phi(q)(\varphi(q)-s)}\frac{1}{s}\left(\frac{(\Phi(q)-s)s}{q-\psi(s)}
+\frac{\Phi(q)-s}{\delta}-\frac{\Phi(q)}{\delta \varphi(q)}(\varphi(q)-s)\right)\\
&=\frac{\delta \varphi(q)}{\Phi(q)(\varphi(q)-s)}\left(
\frac{\Phi(q)-s}{q-\psi(s)}+\frac{\Phi(q)-\varphi(q)}
{\delta \varphi(q)}\right)\\
&=\frac{\delta \varphi(q)}{q(\varphi(q)-s)}\left(\mathbb E\left[e^{s\underline{X}_{e(q)}}\right]-
\mathbb E\left[e^{\varphi(q)\underline{X}_{e(q)}}\right]
\right),
\end{split}
\]
where the final equality follows from (2.4) and the fact of $\psi(\varphi(q))-\delta \varphi(q)=q$ (see the definition of $\varphi(q)$ given by (2.3)). Thus we can obtain
\begin{equation}
F_2(x)=
\frac{\delta \varphi(q)}{q}\int_{(-\infty,x)}e^{-\varphi(q)(x-z)}\mathbb P\left(\underline{X}_{e(q)} \in dz\right),
\end{equation}
since both sides have the same Laplace transform.

It is easy to obtain that $e^{\varphi(q)x}F_2(x)$ is increasing (decreasing) on $(-\infty,0]$ if $\delta >0$ $(\delta <0)$. From (2.8), formula (2.18) follows from the application of integration by parts.
For the derivation of (2.19), we note first that
\[
e^{\varphi(q)x}F_2(x)\leq \frac{\delta \varphi(q)}{q}\int_{(-\infty,0)}e^{\varphi(q)z}\mathbb P\left(\underline{X}_{e(q)} \in dz\right)\leq \frac{\delta \varphi(q)}{q} \mathbb E\left[e^{\varphi(q)\underline{X}_{e(q)}}\right],\ \ \delta >0,
\]
and
\[
e^{\varphi(q)x}F_2(x)\geq \frac{\delta \varphi(q)}{q}\int_{(-\infty,0)}e^{\varphi(q)z}\mathbb P\left(\underline{X}_{e(q)} \in dz\right)\geq \frac{\delta \varphi(q)}{q} \mathbb E\left[e^{\varphi(q)\underline{X}_{e(q)}}\right], \ \ \delta <0,
\]
and then use the result of  $\mathbb E\left[e^{\varphi(q)\underline{X}_{e(q)}}\right]=\frac{q}{\Phi(q)}
\frac{\Phi(q)-\varphi(q)}{-\delta \varphi(q)}$ (which can be derived by using (2.4) and the fact of  $\psi(\varphi(q))-\delta \varphi(q)=q$).

(iii) The result can be obtained from (2.20) and (2.21).

(iv) Recall (2.6). The result follows easily from the following fact (see (2.4)):
\[
\mathbb P\left(\overline{X}_{e(q)}\in dy\right)=\Phi(q)e^{-\Phi(q)y}dy \ \ for \ \ y>0.
\]
\end{proof}

\begin{Remark}
If $\delta <0$, we know that $F_1(x)+1$ increases on $[0,\infty]$ with $F_1(0)+1=\frac{\varphi(q)}{\Phi(q)}$ and $F_1(\infty)+1=1$. So formula (2.20) implies that $F_1(x)+1$ is an infinitely divisible distribution on $[0,\infty)$. Applying integration by parts gives
\begin{equation}
\int_{[0,\infty)} e^{-s x}d\big(F_1(x)+1\big)=\frac{\hat{\mathbb E}\left[e^{-s \overline{Y}_{e(q)}}\right]}
{\mathbb E\left[e^{-s \overline{X}_{e(q)}}\right]}=e^{\int_{0}^{\infty}(e^{-sx}-1)\Pi_1(dx)}, \ \ s\geq 0.
\end{equation}
where
\begin{equation}
\Pi_{1}(dx)=\int_{0}^{\infty}\frac{1}{t}e^{-qt}\left(\hat{\mathbb P}\left(Y_t \in dx\right)-\mathbb P\left(X_t \in d x\right)\right)dt,
\end{equation}
and in the above derivation, we have used the following well-known result (see,
e.g., Theorem 5 on page 160 in~\cite{B}):
\[
\mathbb E\left[e^{-s \overline {X}_{e(q)}}\right]
=e^{\int_{0}^{\infty}\frac{1}{t}e^{-qt}\int_{0}^{\infty}(e^{-s x}-1)
\mathbb P\left(X_t\in dx\right)dt}, \ \ s\geq 0.
\]
From the L\'evy-Khintchine formula (see, e.g., Theorem 8.1
on page 37 in~\cite{Sa}), formula (2.22) means that $\Pi_1(dx)$ is a measure. This result can also be proved by using the following well-known Kendall's identity:
\begin{equation}
\frac{\mathbb P\left(X_t\in dx\right)}{t}dt=\frac{\mathbb P\left(\tau_x^+ \in dt\right)}{x}dx, \ \ t,x>0,
\end{equation}
and writing $\Pi_1(dx)$ as
\[
\Pi_1(dx)=\frac{1}{x}\left(\hat{\mathbb E}\left[e^{-q\hat{\tau}_x^+}\right]-\mathbb E\left[
e^{-q \tau_x^+}\right]\right)dx,
\]
where $\tau_x^+=\inf\{t>0, X_t> x\}$ and $\hat{\tau}_x^+=\inf\{t>0, Y_t>x\}$.
Since $Y_t=X_t-\delta t$ and $\delta<0$, we have $\hat{\mathbb E}\left[e^{-q\hat{\tau}_x^+}\right]-\mathbb E\left[
e^{-q \tau_x^+}\right]\geq 0$ for $x>0$ and $q>0$.
\end{Remark}

\begin{Remark}
The above remark states that $\Pi_1(dx)$ is a measure if $X \in \mathbb S\mathbb N\mathbb L\mathbb P$ and $\delta <0$, and formula (2.24) is an important condition for the proof of this result.
However, formula (2.24) holds only for a spectrally negative L\'evy process $X$ not for a general L\'evy process. Despite of this, we still believe that $\Pi_1(dx)$ is a measure for a general L\'evy process $X$ if $\delta <0$ (whose proof is left to future research).
\end{Remark}

\section{Proofs of Theorem 2.1 and Theorem 2.2}
In this section, we want to prove the main results and give first some preliminary results.

\subsection{Some preliminary results}

Consider the following L\'evy process $X_t$,
\begin{equation}
  X_t = X_0 + \mu t+\sigma W_t -\sum_{k=1}^{N_t}Z_k,
\end{equation}
where $\mu$ and $X_0$ are constants; $(W_t)_{t\geq 0}$ is a standard Brownian motion with $W_0=0$, and $\sigma > 0$ is the volatility of the diffusion; $\sum_{k=1}^{N_t}Z_k$ is a compound Poisson process with intensity $\lambda$ and the following jump distribution:
\begin{equation}
p(z):=\frac{\mathbb P\left(Z_1 \in dz \right)}{dz}=
\sum_{k=1}^{m}\sum_{j=1}^{m_k}d_{kj}\frac{(\vartheta_k)^jz^{j-1}}{(j-1)!}e^{-\vartheta_k z}, \ \ z > 0;
\end{equation}
moreover, $(W_t)_{t\geq 0}$, $(N_t)_{t\geq 0}$ and $\{Z_k; k=1,2,\ldots\}$ are independent mutually.
Note that $\vartheta_k$ and $d_{kj}$ in (3.2) can take complex value.

For the process $X$ given by (3.1), since $\sigma >0$ and the jump part of it is a compound Poisson process, the following lemma holds. For its proof, we refer the reader to, e.g., Theorem 305 in~\cite{St} or Theorem 5.9 and Remark 5.16 in~\cite{MP}.
\begin{Lemma}
If the process $X_t$ in (1.1) is given by (3.1), then there is a unique strong solution $U=(U_t)_{t\geq 0}$ to (1.1).
\end{Lemma}

In~\cite{Zhouthe}, we have derived the following result, see Theorems 4.1 and 4.2 in that paper.
\begin{Lemma}
For $X$ given by (3.1) and the associated unique strong solution $U$ to (1.1), we have
\begin{equation}
\left\{\begin{array}{cc}
\mathbb P_x\left(U_{e(q)}\leq y\right)=
K_q(y-x)+\int_{b-x}^{y-x}F_1(y-x-z)K_q(dz), &  y\geq b,\\
\mathbb P_x\left(U_{e(q)}< y\right)=
K_q(y-x)-\int_{y-x}^{b-x}F_2(y-x-z)K_q(dz), & y\leq b,
\end{array}
\right.
\end{equation}
where $K_q(x)$, $F_1(x)$ and $F_2(x)$ are given respectively by (2.6), (2.7) and (2.8).
\end{Lemma}

A property of $p(z)$ given by (3.2) is that its Laplace transform is a rational function. Particularly, both linear combination of exponential distributions and phase-type distributions have the form of (3.2). Therefore, the following result holds and one can refer to Proposition 1 in~\cite{AAP} for its proof.
\begin{Lemma}
For any $X \in \mathbb S\mathbb N\mathbb L\mathbb P$, there is a sequence of L\'evy process $X_t^n$ having the form of (3.1) such that
\begin{equation}
\lim_{n \uparrow \infty} \sup_{s \in [0,t]}|X^n_s-X_s|=0, \ \ almost \ \ surely.
\end{equation}
\end{Lemma}

In the rest of this section, the process $X$ belongs to $\mathbb S\mathbb N\mathbb L\mathbb P$ and is given and fixed. For this fixed $X \in \mathbb S\mathbb N\mathbb L\mathbb P$, we can find a sequence of $X^n$ with the form of (3.1) such that (3.4) holds (see Lemma 3.3). For each $X^n$, denote by $U^n$ the unique strong solution to (1.1) (see Lemma 3.1).
For the sequence of $U^n$, the following Lemma 3.4 holds, and one can refer to Lemma 12 in~\cite{KL} for its derivation.
\begin{Lemma}
There is a stochastic process $U^{\infty}=(U^{\infty}_t)_{t\geq 0}$ satisfying
\begin{equation}
\lim_{n \uparrow \infty} \sup_{s \in [0,t]}|U^n_s-U^{\infty}_s|=0, \ \ almost \ \ surely.
\end{equation}
\end{Lemma}
In the following, our objective is to show that $U^{\infty}$ is a strong solution to (1.1) for the above given $X \in \mathbb S\mathbb N\mathbb L\mathbb P$, and we note first that the following result holds and refer the reader to Proposition 15 in~\cite{KL} for its derivation.
\begin{Lemma}
For any given $x \in \mathbb R$, if $\mathbb P_x\left(U^{\infty}_t=b\right)=0$ for Lebesgue almost surely every $t \geq 0$, then $U^{\infty}=(U^{\infty}_t)_{t\geq 0}$
is the strong solution to (1.1), i.e.,
\begin{equation}
U^{\infty}_t=X_t-\delta \int_{0}^{t}\textbf{1}_{\{U^{\infty}_s > b\}}ds, \ \ t\geq0,
\end{equation}
\end{Lemma}
Therefore, to derive (3.6), it is enough to show that $\mathbb P_x\left(U^{\infty}_t=b\right)=0$ for Lebesgue almost surely every $t \geq 0$.

\subsection{An important result}
In this subsection, we present an important proposition. First, as the process $X^n$ has the form of (3.1), we can obtain from Lemma 3.2 that
\begin{equation}
\left\{\begin{array}{cc}
\mathbb P_x\left(U^n_{e(q)}\leq y\right)=
K^n_q(y-x)+\int_{b-x}^{y-x}F^n_1(y-x-z)K^n_q(dz), & y\geq b,\\
\mathbb P_x\left(U^n_{e(q)}< y\right)=
K^n_q(y-x)-\int_{y-x}^{b-x}F^n_2(y-x-z)K^n_q(dz), & y\leq b,
\end{array}\right.
\end{equation}
where $K^n_q(x)$ is the convolution of $\underline{Y}^n_{e(q)}$ under $\hat{\mathbb P}$ and $\overline{X}^n_{e(q)}$ under $\mathbb P$,
\begin{equation}
\begin{split}
\int_{0}^{\infty} e^{-s z}F^n_1(z)dz
=\frac{1}{s}\left(\frac{\hat{\mathbb E}\left[e^{-s \overline{Y}^n_{e(q)}}\right]}
{\mathbb E\left[e^{-s \overline{X}^n_{e(q)}}\right]}-1\right), \ \ s > 0,
\end{split}
\end{equation}
and
\begin{equation}
\int_{-\infty}^{0} e^{sz}F^n_2(z)dz
=\frac{1}{s}\left(\frac{\mathbb E\left[e^{s \underline{X}^n_{e(q)}}\right]}
{\hat{\mathbb E}\left[e^{s \underline{Y}^n_{e(q)}}\right]}-1\right), \ \ s > 0.
\end{equation}

In addition, define $\psi^n(\theta)=\ln\left(\mathbb E\left[e^{\theta X^n_1}\right]\right)$ for $\theta \geq 0$. For $q>0$, introduce
\begin{equation}
\Phi^n(q)=\sup\{\theta \geq 0: \psi^n(\theta)=q\}, \ \ \varphi^n(q)=\sup\{\theta \geq 0: \psi^n(\theta)-\delta \theta =q\}.
\end{equation}

For the above two probabilities $\mathbb P_x\left(U^n_{e(q)}\leq y\right)$ and $\mathbb P_x\left(U^n_{e(q)}< y\right)$ given by (3.7), the following proposition holds.

\begin{Proposition}
 For any given $q>0$, it holds that
\begin{equation}
\left\{\begin{array}{cc}
\lim_{n\uparrow \infty}\mathbb P_x\left(U^n_{e(q)} \leq y \right)=K_q(y-x)+\int_{b-x}^{y-x}F_1(y-x-z)K_q(dz), &  y > b,\\
\lim_{n \uparrow \infty}\mathbb P_x\left(U^n_{e(q)}< y\right)=K_q(y-x)-\int_{y-x}^{b-x}F_2(y-x-z)K_q(dz), &   y <b.
\end{array}\right.
\end{equation}
\end{Proposition}

\begin{proof}
We only prove the result for $y>b$ and omit the corresponding details for the case of $y < b$ since they are similar.

Assume that $y>b$. It is known that (see, e.g., Lemma 13.4.1 in~\cite{Wh})
\begin{equation}
|\overline{X}^n_t-\overline{X}_t|\leq \sup_{0 \leq s \leq t}|X^n_s-X_s|\ and \ |\underline{Y}^n_t-\underline{Y}_t|\leq \sup_{0 \leq s \leq t}|Y^n_s-Y_s|.
\end{equation}
For $s>0$, formulas (3.4) and (3.12) yield
\begin{equation}
\lim_{n \uparrow \infty}\mathbb E\left[e^{-s \overline{X}^n_{e(q)}}\right]=\mathbb E\left[e^{-s \overline{X}_{e(q)}}\right] \ and \ \lim_{n \uparrow \infty}\hat{\mathbb E}\left[e^{s \underline{Y}^n_{e(q)}}\right]=\hat{\mathbb E}\left[e^{s \underline{Y}_{e(q)}}\right],
\end{equation}
which means that $\overline{X}^n_{e(q)}$ and $\underline{Y}^n_{e(q)}$ converge respectively to $\overline{X}_{e(q)}$ and $\underline{Y}_{e(q)}$ in distribution. And since $K_q(x)$ is continuous on $(-\infty,\infty)$ (see Theorem 2.3 (iv)), we have
\begin{equation}
\lim_{n \uparrow \infty} K_q^n(x) =K_q(x), \ \ x  \in  \mathbb R.
\end{equation}

For $y>b$, we can write $\int_{b-x}^{y-x}F^n_1(y-x-z)K^n_q(dz)$  as
\begin{equation}
\mathbb E\left[F_1^n(y-x-Z_1^n)\textbf{1}_{\{b-x<Z_1^n < y-x\}}\right],
\end{equation}
where the law of $Z_1^n$ is given by $K_q^n(z)$. The following Lemma 3.6 (i) gives that if $x_n \rightarrow x >0$, then
\begin{equation}
\lim_{n \uparrow \infty}F_1^n(x_n)=F_1(x).
\end{equation}
Besides, as $F_1^n(x)$ for $n=1,2,\ldots$, are uniformly bounded (see Lemma 3.6 (ii)), we deduce from (3.14), (3.15), (3.16) and the bounded convergence theorem that
\[
\lim_{n \uparrow \infty} \mathbb E\left[F_1^n(y-x-Z_1^n)\textbf{1}_{\{b-x<Z_1^n < y-x\}}\right]=\mathbb E\left[F_1(y-x-Z_1)\textbf{1}_{\{b-x<Z_1 < y-x\}}\right],
\]
where the distribution of $Z_1$ is given by $K_q(z)$.
From (3.14) and the last formula, letting $n \uparrow \infty$ in (3.7) will yield the result in (3.11) for $y>b$.
\end{proof}

\begin{Lemma}
For the continuous functions $F_1(x)$ in (2.7) and $F_1^n(x)$ in (3.8), we have the following results.

(i) $F_1^n(x)$ is uniformly convergence to $F_1(x)$ on $[0,\infty]$.

(ii) $F_1(x)$, $F_1^1(x), F_1^2(x)$, $\ldots$, are uniformly bounded.
\end{Lemma}

\begin{proof}
Similar to (3.13), we can also deduce that
\begin{equation}
\lim_{n \uparrow \infty}\mathbb E\left[e^{s\underline{X}^n_{e(q)}}\right]=\mathbb E\left[e^{s \underline{X}_{e(q)}}\right] \ and \ \lim_{n \uparrow \infty}\hat{\mathbb E}\left[e^{-s \overline{Y}^n_{e(q)}}\right]=\hat{\mathbb E}\left[e^{-s \overline{Y}_{e(q)}}\right],  \ \ s>0.
\end{equation}
In addition, formulas (2.4), (2.5) and (3.13) give
\[
\lim_{n \uparrow \infty}\Phi^n(q)=\Phi(q) \ \ \ and \ \ \lim_{n \uparrow \infty}\varphi^n(q)=\varphi(q).
\]

For  $s>0$, it follows from (2.16), (3.13) and (3.17) that
\begin{equation}
\begin{split}
&\lim_{n \uparrow \infty}\int_{0}^{\infty} e^{-s x}d F_1^n(x)=\lim_{n \uparrow \infty}
\left(\frac{\hat{\mathbb E}\left[e^{-s \overline{Y}^n_{e(q)}}\right]}
{\mathbb E\left[e^{-s \overline{X}^n_{e(q)}}\right]}-\frac{\varphi^n(q)}{\Phi^n(q)}\right)\\
&=\frac{\hat{\mathbb E}\left[e^{-s \overline{Y}_{e(q)}}\right]}
{\mathbb E\left[e^{-s \overline{X}_{e(q)}}\right]}-\frac{\varphi(q)}{\Phi(q)}=\int_{0}^{\infty} e^{-s x}dF_1(x),
\end{split}
\end{equation}

As $F_1(x)$ $\big(F_1^n(x)\big)$ is monotone, bounded and continuous on $[0,\infty)$ (see Theorem 2.3),  we can think simply $F_1(x)$ $\big(F_1^n(x)\big)$ or $F_1(0)-F_1(x)$ $\big(F_1^n(0)-F_1^n(x)\big)$ as a measure. So from (3.18) and the continuity theorem for Laplace transforms (see, e.g., Theorem 2a on page 433 in~\cite{F}), we have
\begin{equation}
\lim_{n \uparrow \infty} F_1^n(x)=F_1(x) \ \ or \ \
\lim_{n \uparrow \infty} \big(F_1^n(0)-F_1^n(x)\big)=F_1(0)-F_1(x), \ \ x > 0.
\end{equation}

Note that (see (2.12))
\[
\lim_{n\uparrow \infty}F_1^n(0)=\lim_{n\uparrow \infty}\frac{\varphi^n(q)-\Phi^n(q)}{\Phi^n(q)}
=\frac{\varphi(q)-\Phi(q)}{\Phi(q)}=F_1(0),
\]
and (see (2.9))
\[
F^n_1(\infty)=F_1(\infty)=0.
\]
The above three formulas yield
\[
\lim_{n\uparrow \infty}F_1^n(x) = F_1(x), \ \ for \ \ all \ \  x \in [0,\infty].
\]
Since $F_1^n(x)$ and $F_1(x)$ are  monotone, bounded and continuous on $[0,\infty)$, we obtain  that $F_1^n(x)$ is uniformly convergence to $F_1(x)$ on $[0,\infty]$.

In addition, we have (see (2.17))
\[
-1\leq F_1^n(x), F_1(x) \leq \sup_{n\geq 1}{|\frac{\varphi^n(q)-\Phi^n(q)}{\Phi^n(q)}|}<\infty.
\]
\end{proof}

\subsection{Proof of Theorem 2.1}
We divide the proof of Theorem 2.1 into three steps, where in the first step, we use some ideas given in the proofs of Proposition 15 and Lemma 21 in~\cite{KL}.

Step 1. From Lemma 3.5, to prove that $U^{\infty}$ is a strong solution to (1.1) is equivalent to show that $\mathbb P_x\left(U_t^{\infty}=b\right)=0$ for Lebesgue almost every $t>0$.

Recall Lemma 3.4.
For any $\varepsilon > 0$, it is obvious that
\begin{equation}
\{U_t^{\infty}= b\} \subseteq \{U^n_t \in (b-\varepsilon,b+\varepsilon)\ \ eventually \ \ as \ \ n\uparrow\infty\}.
\end{equation}
It follows from standard measure theory that
\begin{equation}
\mathbb P_x\left(U_t^{\infty}=b\right)\leq \underline{\lim}_{n \uparrow \infty}\mathbb P_x\left(b- \epsilon < U^n_t < b+\epsilon \right).
\end{equation}
Now Fatou's lemma leads to that
\begin{equation}
\begin{split}
&\int_{0}^{\infty} q e^{-q t}\mathbb P_x\left(U_t^{\infty}=b\right)dt=\mathbb P_x\left(U^{\infty}_{e(q)}=b\right)\\
& \leq \underline{\lim}_{n \uparrow \infty}\mathbb P_x\left(b- \epsilon \leq U^n_{e(q)}  \leq b+\epsilon \right)\\
&=\underline{\lim}_{n \uparrow \infty} \left(\mathbb P_x\left(U^n_{e(q)} \leq  b+\epsilon \right)-\mathbb P_x\left(U^n_{e(q)} < b-\epsilon \right)\right).
\end{split}
\end{equation}

Step 2. It follows from (3.11) that
\begin{equation}
\begin{split}
&\lim_{n \uparrow \infty}\left(\mathbb P_x\left(U^n_{e(q)} \leq b+\epsilon \right)-\mathbb P_x\left(U^n_{e(q)} < b-\epsilon \right)\right)\\
&=K_q(b+\epsilon-x)+\int_{b-x}^{b+\epsilon-x}F_1(b+\epsilon-x-z)K_q(dz)\\
&-
K_q(b-\epsilon-x)+\int_{b-\epsilon-x}^{b-x}F_2(b-\epsilon-x-z)K_q(dz).
\end{split}
\end{equation}

Obviously, we have (since Theorem 2.3 (iv) holds)
\begin{equation}
\begin{split}
&\lim_{\varepsilon \downarrow 0}\big(K_q(b+\epsilon-x)-K_q(b-\epsilon-x)\big)=0.
\end{split}
\end{equation}

Recalling the fact that both $F_1(x)$ and $F_2(-x)$ are continuous and bounded on $[0,\infty]$ (see Theorem 2.3), we can derive
\begin{equation}
\begin{split}
&\lim_{\varepsilon \downarrow 0} \int_{b-x}^{b+\epsilon-x}F_1(b+\epsilon-x-z)K_q(dz)=0
=\lim_{\varepsilon \downarrow 0}\int_{b-\epsilon-x}^{b-x}F_2(b-\epsilon-x-z)K_q(dz).
\end{split}
\end{equation}

Therefore, for any given $q>0$, we obtain from (3.22) $\sim$ (3.25) that
\begin{equation}
\begin{split}
&\mathbb P_x\left(U^{\infty}_{e(q)}=b\right)
 \leq \lim_{\varepsilon \downarrow 0} \lim_{n \uparrow \infty}\mathbb P_x\left(b- \epsilon < U^n_{e(q)}  < b+\epsilon \right)=0,
\end{split}
\end{equation}
which gives us the desired result that $\mathbb P_x\left(U_t^{\infty}=b\right)=0$ for Lebesgue almost every $t>0$.

Step 3. The proof of $\mathbb P_x\left(U^{\infty}_{e(q)}=y\right)=0$ for $y \neq b$ is similar to that of (3.26).  For example, for $y > b$ and $0< \varepsilon < y-b$,
similar to deriving (3.22), we have

\begin{equation}
\begin{split}
&\mathbb P_x\left(U^{\infty}_{e(q)}=y\right)
\leq  \lim_{n \uparrow \infty}\left(\mathbb P_x\left(U^n_{e(q)} \leq y+\epsilon \right)- \mathbb P_x\left(U^n_{e(q)} \leq y-\epsilon \right)\right).
\end{split}
\end{equation}
Since $y-\epsilon >b$. From (3.11), we can write the right-hand side of (3.27) as
\begin{equation}
\begin{split}
&K_q(y+\epsilon-x)+\int_{b-x}^{y+\epsilon-x}F_1(y+\epsilon-x-z)K_q(dz)\\
&-K_q(y-\epsilon-x)-\int_{b-x}^{y-\epsilon-x}F_1(y-\epsilon-x-z)K_q(dz),
\end{split}
\end{equation}
which combined with Theorem 2.3 and  the bounded convergence theorem, gives $\mathbb P_x\left(U^{\infty}_{e(q)}=y\right)=0$ for $y > b$ after letting $\varepsilon \downarrow 0$ in (3.28).

\subsection{Proof of Theorem 2.2}
From the above subsection, we know that $U^{\infty}$ is a strong solution to (1.1) and write $U^{\infty}$ as $U$ for the sake of brevity. From (2.14), we have $\mathbb P\left(U_{t}=z\right)=0$ for Lebesgue almost every $t>0$, which combined with (3.5), leads to
\[
\lim_{n \uparrow \infty}\mathbb P_x\left(U^{n}_{e(q)} > z\right)=
\mathbb P_x\left(U_{e(q)}> z\right), \lim_{n \uparrow \infty}\mathbb P_x\left(U^{n}_{e(q)} < z\right)=
\mathbb P_x\left(U_{e(q)}< z\right), \ \ z \in \mathbb R.
\]
The last formula and Proposition 3.1 will lead to
\[
\left\{\begin{array}{cc}
\mathbb P_x\left(U_{e(q)}\leq y\right)=K_q(y-x)+\int_{b-x}^{y-x}F_1(y-x-z)K_q(dz), &  y \geq b,\\
\mathbb P_x\left(U_{e(q)}< y\right)=K_q(y-x)-\int_{y-x}^{b-x}F_2(y-x-z)K_q(dz), &   y\leq b,
\end{array}\right.
\]
from which Theorem 2.2 is derived.

\section{Conclusion Remarks}
In this paper, we use a different approach to derive the following results: there is a strong solution $U$ to (1.1) if $X$ is a spectrally negative L\'evy process with some additional conditions satisfied and the expression of the q-potential measures of $U$ can be written in terms of Winer-Hopf factors. In the following, some remarks and discussions on the extension of the above two results are presented.
In this section, the process $X$ in (1.1) is a general L\'evy process. And for a general L\'evy process $X$, the three functions $K_q(x)$, $F_1(x)$ and $F_2(x)$ are defined still by (2.6), (2.7) and (2.8), respectively.

Consider first the following jump diffusion process $X$:
\begin{equation}
  X_t = X_0 + \mu t+\sigma W_t +\sum_{k=1}^{N^+_t}Z^+_k-\sum_{k=1}^{N^-_t}Z^-_k,
\end{equation}
where $\sigma >0$ and the probability density functions of $Z^+_1$ and $Z_1^-$ are given respectively by
\begin{equation}
p^+(z)=\frac{\mathbb P\left(Z^+_1 \in dz\right)}{dz}=\sum_{k=1}^{m^+}\sum_{j=1}^{m_k}c_{kj}\frac{(\eta_k)^jz^{j-1}}{(j-1)!}e^{-\eta_k z}, \ \ z > 0,
\end{equation}
and
\begin{equation}
p^-(z)=\frac{\mathbb P\left(Z^-_1 \in dz\right)}{dz}=\sum_{k=1}^{n^-}\sum_{j=1}^{n_k}d_{kj}\frac{(\vartheta_k)^jz^{j-1}}{(j-1)!}e^{-\vartheta_k z}, \ \ z > 0.
\end{equation}

For any given L\'evy process $X$, a similar result to Lemma 3.3 also holds (see Proposition 1 in~\cite{AAP}), i.e., we can find a sequence of $X^n$ having the form of (4.1) such that
\begin{equation}
\lim_{n \uparrow \infty} \sup_{s \in [0,t]}|X^n_s-X_s|=0, \ \ almost \ \ surely.
\end{equation}

For any given and fixed L\'evy process $X$, denote by $X^n$, $n\geq 1$, the corresponding sequence such that (4.4) holds. For each $X^n$, Lemma 3.1 also holds and we denote by $U^n$ the corresponding unique strong solution to (1.1). After looking carefully the proof of Lemma 12 in~\cite{KL}, we find that Lemma 12 in~\cite{KL} (i.e., Lemma 3.4) holds also for the above process $U^n$, i.e., there is a process $U^{\infty}$ such that
\begin{equation}
\lim_{n \uparrow \infty} \sup_{s \in [0,t]}|U^n_s-U_s^{\infty}|=0, \ \ almost \ \ surely.
\end{equation}
Besides, in~\cite{Zhouthe}, we have showed that Lemma 3.2 holds also for the above $X^n$ and $U^n$, i.e.,
\begin{equation}
\left\{\begin{array}{cc}
\mathbb P_x\left(U^n_{e(q)}\leq y\right)=
K^n_q(y-x)+\int_{b-x}^{y-x}F^n_1(y-x-z)K^n_q(dz), & y\geq b,\\
\mathbb P_x\left(U^n_{e(q)}< y\right)=
K^n_q(y-x)-\int_{y-x}^{b-x}F^n_2(y-x-z)K^n_q(dz), & y\leq b.
\end{array}\right.
\end{equation}

Denote by $\mathbb L$ the space of all L\'evy processes such that the following conditions holds.\\
(i)
\begin{equation}
\lim_{s \uparrow \infty}\frac{\hat{\mathbb E}\left[e^{-s \overline{Y}_{e(q)}}\right]}
{\mathbb E\left[e^{-s \overline{X}_{e(q)}}\right]}< \infty \ \ and \ \ \lim_{s \uparrow \infty}\frac{\mathbb E\left[e^{s \underline{X}_{e(q)}}\right]}
{\hat{\mathbb E}\left[e^{s \underline{Y}_{e(q)}}\right]}< \infty \ \ if \ \ \delta>0.
\end{equation}
(ii) $K_q(x)$ is continuous  on  $(-\infty, \infty)$.\\
(iii) $F_1(x)$ ($F_2(-x)$) is continuous and bounded on $[0,\infty)$.\\
(iv) $F_1^n(x)$ ($F^n_2(-x)$) is uniformly convergence to $F_1(x)$ ($F_2(-x)$) on $(0,\infty)$.\\
(v)  For $i=1,2$, $F_i(x)$, $F_i^1(x), F_i^2(x)$, $\ldots$, are uniformly bounded.

From (2.10), (2.11), Theorem 2.3 and Lemma 3.6, we know $\mathbb S \mathbb N\mathbb L\mathbb P$ is a subset of $\mathbb L$. In the following, we intend to illustrate that Theorems 2.1 and 2.2 hold for $X \in \mathbb L$.

First, applying (4.6), (ii), (iv), (v) and the result of $\lim_{n \uparrow \infty} K_q^n(x) =K_q(x)$ for $x \in \mathbb R$ (whose proof is the same as that of (3.14) since condition (ii) holds) gives that formula (3.11) also holds for the above process $U^n$, i.e.,
\[
\left\{\begin{array}{cc}
\lim_{n \uparrow \infty}\mathbb P_x\left(U^n_{e(q)}\leq y\right)=
K_q(y-x)+\int_{b-x}^{y-x}F_1(y-x-z)K_q(dz), & y>b,\\
\lim_{n\uparrow \infty}\mathbb P_x\left(U^n_{e(q)}< y\right)=
K_q(y-x)-\int_{y-x}^{b-x}F_2(y-x-z)K_q(dz), & y<b.
\end{array}\right.
\]
Then from (ii), (iii) and the last formula, we can derive that $\mathbb P_x\left(U_t^{\infty}=b\right)=0$  for Lebesgue almost  every  $t>0$ (see Step 2 in Section 3.3), thus $U^{\infty}$ is a strong solution to (1.1) (see Lemma 3.5).
Similar derivations will derive $\mathbb P_x\left(U_{e(q)}^{\infty}=y\right)=0$ for $ y\neq b$ (see Step 3 in Section 3.3). Finally, the proof that Theorem 2.2 holds also for $X \in \mathbb L$ is almost the same as that in Section 3.4.

Therefore, the most important part of future researches is to show that which L\'evy process belongs to $\mathbb L$. To close this paper, we present some remarks on the above five conditions.
\begin{Remark}
The condition (4.7) is important for the existence of a strong solution to (1.1), and it is possible that  equation (1.1) does not exist a solution if this condition is failed, which will be illustrated in the following. Assume that $X$ is a compound Poisson process with positive drift, i.e.,
\begin{equation}
X_t=X_0+\mu t +\sum_{i=1}^{N_t} Z_i, \ \ t\geq 0,
\end{equation}
where $\mu > 0$. For $\delta > \mu >0$, we have
\[
Y_t=X_t-\delta t=X_0+(\mu-\delta) t +\sum_{i=1}^{N_t} Z_i, \ \ t\geq 0.
\]
Since $\mu>0$ and $\mu-\delta < 0$, we have $\hat{\mathbb P}\left(\overline{Y}_{e(q)}=0\right)> 0$ and $\mathbb P\left(\overline{X}_{e(q)}=0\right)=0$, thus
$
\lim_{s \uparrow \infty}\frac{\hat{\mathbb E}\left[e^{-s \overline{Y}_{e(q)}}\right]}
{\mathbb E\left[e^{-s \overline{X}_{e(q)}}\right]}=\infty,
$
which means that formula (4.7) is false. To this end, it is easy to see that there is no solution to (1.1) with $X$ given by (4.8) and $X_0=b$.

In addition, this condition is necessary for the boundedness of $F_1(x)$ since
\[
F_1(0)=\lim_{s\uparrow \infty}\int_{0}^{\infty}se^{-sx}F_1(x)dx=\lim_{s\uparrow \infty}\left(\frac{\hat{\mathbb E}\left[e^{-s \overline{Y}_{e(q)}}\right]}
{\mathbb E\left[e^{-s \overline{X}_{e(q)}}\right]}-1\right).
\]
\end{Remark}

\begin{Remark}
For the second condition, if $0$ is regular for $(0,\infty)$ or $(-\infty,0)$, i.e., $\mathbb P\left(\tau_0^+=0\right)=1$ with $\tau_0^+=\inf\{t>0, X_t> 0\}$ or $\hat{\mathbb P}\left(\hat{\tau}_0^-=0\right)=1$ with $\hat{\tau}_0^-=\inf\{t>0, Y_t<0\}$, then
$\mathbb P\left(\overline{X}_{e(q)}=z\right)=0$ or $\hat{\mathbb P}\left(\underline{Y}_{e(q)}=z\right)=0$  for all  $z \in \mathbb R$,
which means that $K_q(x)$ is continuous (see (2.6)). From Theorem 6.5 in~\cite{K}, we know that condition (ii)
holds for almost all L\'evy process. For example, if $X$ is a spectrally negative L\'evy process or has paths of unbounded variation, then $0$ is regular for $(0,\infty)$ and thus $K_q(x)$ is continuous. If $\delta >0$, Theorem 6.5 in~\cite{K} gives that either
$\mathbb P\left(\tau_0^+=0\right)=1$ or $\hat{\mathbb P}\left(\hat{\tau}_0^-=0\right)=1$ must hold. So the second condition holds for all L\'evy process if $\delta >0$.
\qed
\end{Remark}

\begin{Remark}
A situation which makes condition (ii) incorrect is that $X$ has paths of bounded variation with negative  drift term meanwhile $Y$ has positive drift term, i.e.,
\[
X_t=d t +J_t \ \ and \ \ Y_t=(d-\delta)t +J_t,
\]
where $\delta <d< 0$ and $J_t$ is a pure jump L\'evy process with bounded variation paths. Because in this situation, both $\mathbb P\left(\overline{X}_{e(q)}=0\right)$ and $\hat{\mathbb P}\left(\underline{Y}_{e(q)}=0\right)$ are positive. So we can think that the first condition  puts some limitations on the case of $\delta >0$ while the second one
gives some limitations on the case of $\delta <0$.
\qed
\end{Remark}

\begin{Remark}
For the remaining three conditions (particularly the last two): (iii), (iv) and (v), at present we can not find an easy approach to show what kind of L\'evy process satisfies them. An potential research direction is considering the quantity $\Pi_1(dx)$ given by (2.23). For example, for $\delta <0$, if we can prove that $\Pi_1(dx)$ is a measure for a general L\'evy process $X$, then we will obtain that $F_1(x)+1$ is an infinitely divisible distribution (see (2.22)) and derive that $F_1^n(x)$ is uniformly convergence to $F_1(x)$ on $(0,\infty)$ (see the derivation of Lemma 3.6). In fact, if $\Pi_1(dx)$ is a measure for a general L\'evy process when $\delta <0$, then we can use similar ideas in~\cite{wuetal} to prove that conditions (iv) and (v) hold for a general L\'evy process and any given $\delta \in \mathbb R$.
However, to show that $\Pi_1(dx)$ is a measure if $\delta <0$ is a difficult work (which has not been considered before to the best of our knowledge) and is left to future research.
\end{Remark}

%

\end{document}